\documentclass[11pt, reqno]{amsart}
\usepackage[T1]{fontenc}
\usepackage{amsfonts}
\usepackage{amsmath}
\usepackage{amsthm}
\usepackage{amssymb}
\usepackage{geometry}
\usepackage{graphicx}
\usepackage{xcolor}
\usepackage{mathtools}
\usepackage[colorlinks=true, linkcolor=blue]{hyperref}
\usepackage{cleveref}
\usepackage[english]{babel}
\usepackage{lineno}
\usepackage{float}

\newtheorem{theorem}{Theorem}[section]

\newtheorem{corollary}[theorem]{Corollary}

\newtheorem{conjecture}[theorem]{Conjecture}
\newtheorem{proposition}[theorem]{Proposition}

\usepackage{tikz}
\usetikzlibrary{positioning}
\usetikzlibrary{decorations,arrows}
\usetikzlibrary{decorations.markings}
\numberwithin{equation}{section}

\usepackage{rustic}
\usepackage[T1]{fontenc}

\emergencystretch=\maxdimen
\title{A biased edge coloring game}
\author{Runze Wang}
\address[]{Department of Mathematical Sciences, University of Memphis, Memphis, TN 38152, USA}
\email{runze.w@hotmail.com}
\thanks{}
\date{\today}
\subjclass[]{}

\begin{document}

\sloppy

\begin{abstract}
    We combine the ideas of edge coloring games and asymmetric graph coloring games and define the \emph{$(m,1)$-edge coloring game}, which is alternatively played by two players Maker and Breaker on a finite simple graph $G$ with a set of colors $X$. Maker plays first and colors $m$ uncolored edges on each turn. Breaker colors only one uncolored edge on each turn. They make sure that adjacent edges get distinct colors. Maker wins if eventually every edge is colored; Breaker wins if at some point, the player who is playing cannot color any edge. We define the \emph{$(m,1)$-game chromatic index} of $G$ to be the smallest nonnegative integer $k$ such that Maker has a winning strategy with $|X|=k$. We give some general upper bounds on the $(m,1)$-game chromatic indices of trees, determine the exact $(m,1)$-game chromatic indices of some caterpillars and all wheels, and show that larger $m$ does not necessarily give us smaller $(m,1)$-game chromatic index.
\end{abstract}

\maketitle

\section{Introduction}
In this paper, every graph we study is a finite simple graph.

Bodlaender \cite{Bo} introduced the \emph{coloring game} and the \emph{game chromatic number} in 1991. The coloring game is alternatively played by two players, Maker and Breaker, on a graph $G$ with a set of colors $X$. Maker plays first. On each turn, a player colors an uncolored vertex in $G$ with a color from $X$, while making sure that adjacent vertices get distinct colors. Maker wins if eventually every vertex is colored; Breaker wins if at some point, the player who is playing cannot color any vertex. The game chromatic number of $G$, denoted by $\chi_g(G)$, is defined to be the smallest nonnegative integer $k$ such that Maker has a winning strategy with $|X|=k$. The coloring game has been extensively studied (see \cite{CDDFFG,DZ,EFM,FKKT}) since it was defined. In particular, there is an asymmetric variant of the coloring game --- the \emph{$(a,b)$-coloring game} introduced by Kierstead \cite{Ki}, where Maker colors $a$ uncolored vertices on each turn, and Breaker colors $b$ uncolored vertices on each turn.

The \emph{edge coloring game} was first introduced by Lam, Shiu, and Xu \cite{LSX} in 1999, then also independently introduced by Cai and Zhu \cite{CZ} in 2001. In the edge coloring game, instead of coloring a vertex, a player colors an uncolored edge each time, while making sure that adjacent edges get distinct colors.  Maker wins if eventually every edge is colored; Breaker wins if at some point, the player who is playing cannot color any edge. The \emph{game chromatic index} of $G$, denoted by $\chi'_g(G)$, is the smallest nonnegative integer $k$ such that Maker has a winning strategy with $k$ colors to use. The edge coloring game has also been broadly studied (see \cite{AHS,BBFP,CN,Ke}).

In this paper, we combine the ideas of edge coloring games and asymmetric graph coloring games and define the \emph{$(m,1)$-edge coloring game}, which has a similar setup to the edge coloring game. However, in the $(m,1)$-edge coloring game, Maker colors $m$ uncolored edges each time, and Breaker still colors one uncolored edge each time. And the \emph{$(m,1)$-game chromatic index} of $G$, denoted by $\chi'_g(G;m,1)$, is defined to be the smallest nonnegative integer $k$ such that Maker has a winning strategy with $k$ colors to use. The $(m,1)$-edge coloring game can be seen as a biased variant of the edge coloring game, and they are the same if we take $m=1$.

Let $\Delta(G)$ be the maximum vertex degree of $G$ and let $\chi'(G)$ be the chromatic index of $G$. We can make the following observation.

\begin{proposition} \label{trivial}
    Let $G=(V,E)$ be a graph with $\Delta(G)\ge 1$ and let $m$ be a positive integer. We have
    \begin{align*}
        \Delta(G)\le \chi'(G)\le \chi'_g(G;m,1)\le \max_{(u,v)\in E}\{\deg(u)+\deg(v)-1\}\le 2\Delta(G)-1.
    \end{align*}
\end{proposition}

The first, second, and fourth inequalities are trivial, and we have the third inequality because any edge $e\in E$ is adjacent to at most $\max_{(u,v)\in E}\{\deg(u)+\deg(v)-2\}$ edges, so if we have $\max_{(u,v)\in E}\{\deg(u)+\deg(v)-1\}$ colors, then we can always find a feasible color for $e$, which is different from all the colors of its adjacent edges.

For paths and cycles, it is easy to see that $\chi'_g(P_n;m,1)=\chi'_g(C_n;m,1)=3$ if $n$ is sufficiently large.

Our further work is as follows: 
\begin{itemize}
    \item In Section 2, we give some general upper bounds on the $(m,1)$-game chromatic indices of trees.
    \item In Section 3, we determine the exact $(m,1)$-game chromatic indices of some caterpillars.
    \item In Section 4, we determine the exact $(m,1)$-game chromatic indices of wheels, which show that we do not always have $\chi'_g(G;m_1,1)\le \chi'_g(G;m_2,1)$ for $m_1>m_2$.
    \item In Section 5, we give some remarks and make a conjecture.
\end{itemize}

\section{Trees}

For trees, the following upper bound can be proved in the same way as Lam, Shiu, and Xu proved Theorem 1 in \cite{LSX}. For completeness, we restate their proof here in a more rigorous way.

\begin{theorem} \label{trees}
    Let $T$ be a tree and let $m$ be a positive integer. We have
    \begin{align*}
        \chi'_g(T;m,1)\le \Delta(T)+2.
    \end{align*}
\end{theorem}

\begin{proof}
    Firstly, we can find an edge $u_1 u_2$ with $u_1$ being a leaf, and regard $T$ as a directed rooted tree, with $u_1$ being the root and each arc pointing away from $u_1$. An arc between $v_1$ and $v_2$ will be called $v_1 v_2$ if the arrow points from $v_1$ to $v_2$. On the first turn, Maker can use at most $\Delta(T)$ colors to color a subtree of $T$ with $m$ arcs including $u_1 u_2$. We let $\mathcal{T}$ be this subtree, so now each arc in $\mathcal{T}$ has been colored.

    For each uncolored arc $v_1 v_2$, there are at most $\Delta(T)-1$ arcs adjacent to $v_1 v_2$ at $v_1$, all of which may have been colored, and we will make sure that if two arcs adjacent to $v_1 v_2$ at $v_2$ are colored, then $v_1 v_2$ will immediately get colored. This means when $v_1 v_2$ gets colored, at most $\Delta(T)-1+2=\Delta(T)+1$ arcs adjacent to $v_1 v_2$ have been colored, so $\Delta(T)+2$ colors are always enough. It suffices to give Maker a strategy and update $\mathcal{T}$ according to what players do to make sure:
    
    \qquad \#1. If an arc is colored, then it is in $\mathcal{T}$.
    
    \qquad \#2. If $v_1 v_2$ is in $\mathcal{T}$, and $v_2$ has two other neighbors in $\mathcal{T}$, then $v_1 v_2$ is colored.

    Note that when Maker or Breaker colors an arc $v_1 v_2$ not in $\mathcal{T}$, we have that none of the arcs adjacent to $v_1 v_2$ at $v_2$ has been colored. This is because if $v_2 v_3$ is colored, then by \#1, $v_2 v_3$ is in $\mathcal{T}$, but $\mathcal{T}$ is a directed subtree with root $u_1$, so any arc on the path $u_1 v_2$ is also in $\mathcal{T}$, contradicting the assumption that $v_1 v_2$ is not in $\mathcal{T}$. So when $v_1 v_2$ is colored, it is adjacent to at most $\Delta(T)-1$ colored arcs, all of which are adjacent to $v_1 v_2$ at $v_1$.

    Our current $\mathcal{T}$ fulfills both \#1 and \#2.

    Each time Breaker colors an arc $w_1 w_2$:
    \begin{itemize}
        \item If $w_1 w_2$ is in $\mathcal{T}$, then by \#1 and \#2 we know that at most one arc adjacent to $w_1 w_2$ at $w_2$ has been colored, so $w_1 w_2$ is adjacent to at most $\Delta(T)$ colored arcs.
        \item If $w_1 w_2$ is not in $\mathcal{T}$, then as we just explained, $w_1 w_2$ is adjacent to at most $\Delta(T)-1$ colored arcs.
    \end{itemize}
    
    We find the directed path $P$ from $u_1$ to $w_2$, let $x_1 x_2$ be the last arc that $\mathcal{T}$ and $P$ have in common, and update $\mathcal{T}$ to $\mathcal{T}\cup P$. In this process, we get at most one more arc $v_1 v_2$ in (new) $\mathcal{T}$ with $v_2$ having two other neighbors in (new) $\mathcal{T}$, and if there is such an arc, then it must be $x_1 x_2$. Now we have two cases.

    \begin{itemize}
        \item If $x_1 x_2$ is not colored yet, then Maker can use the first move to color $x_1 x_2$, and use the remaining $m-1$ moves to color other arcs in $\mathcal{T}$. If there are more moves after every arc in $\mathcal{T}$ is colored, then Maker can use these moves to color arcs not in $\mathcal{T}$ but adjacent to $\mathcal{T}$, and each time Maker colors an arc, we will update $\mathcal{T}$ by adding this arc to $\mathcal{T}$.

        \item If $x_1 x_2$ is already colored, then Maker skips coloring $x_1 x_2$, and uses all $m$ moves to do the same thing as what we have in the previous case.
    \end{itemize}

    We can see that, in this algorithm, both \#1 and \#2 are always fulfilled. Also, when we repeat this algorithm, $\mathcal{T}$ will span and finally become $T$, so eventually every edge in $T$ can be colored.
\end{proof}

\begin{corollary}
    Let $T$ be a tree and let $m$ be a positive integer. If $m\ge diam(T)-2$, where $diam(T)$ is the diameter of $T$, then
    \begin{align*}
        \chi'_g(T;m,1)\le \Delta(T)+1.
    \end{align*}
\end{corollary}

\begin{proof}
    We can make the same setup as what we have in the proof of Theorem \ref{trees}. But in this case, every time Breaker colors an arc $w_1 w_2$, Maker can color the whole directed path from $u_2$ to $w_1$ following the arrow direction, because this path has length at most $diam(T)-2\le m$. When an arc $v_1 v_2$ in this path gets colored, there are at most $\Delta(T)-1$ colored arcs adjacent to $v_1 v_2$ at $v_1$, and at most one colored arc (which can only be $w_1 w_2$) adjacent to $v_1 v_2$ at $v_2$. So $\Delta(T)+1$ colors are enough.
\end{proof}

\section{Caterpillars}

A \emph{caterpillar} is a tree, in which we can find a path called the \emph{spine}, such that a vertex is on the spine if and only if the degree of this vertex is at least two. The spine is an induced subgraph of the caterpillar.

\begin{theorem}
    Let $T$ be a caterpillar, let $S=(V(S),E(S))$ be the spine of $T$, and let $m\ge 2$ be an integer. If $\Delta(T)\ge 4$, that is, we have $deg(v)\ge 4$ for some vertex $v\in V(S)$, then
    \begin{align*}
        \chi'_g(T;m,1)=\Delta(T).
    \end{align*}
\end{theorem}

\begin{proof}
    By Proposition \ref{trivial}, we have $\chi'_g(T;m,1)\ge \Delta(T)$. To show that $\chi'_g(T;m,1)\le \Delta(T)$, we need to give Maker a winning strategy with $\Delta(T)$ colors. And it suffices to prove that if we have $\Delta(T)$ colors, then whatever Breaker does, Maker can get all edges in $E(S)$ colored. This is because any edge not in $E(S)$ is adjacent to at most $\Delta(T)-1$ edges, so we can always find a feasible color for it as we have $\Delta(T)$ colors.

    If $|E(S)|\le m$, then Maker can properly color all edges in $E(S)$ using two colors on the first turn.
    
    Now let us assume $|E(S)|\ge m+1$. On the first turn, Maker can randomly choose $m$ edges in $E(S)$ and properly color them with two colors. Then, on each turn, Breaker will either color an edge not in $E(S)$ or an edge in $E(S)$. 

    If Breaker colors an edge $e\notin E(S)$, then Maker will first color the edges in $E(S)$ that are adjacent to $e$. In fact, $e$ is adjacent to one or two edges in $E(S)$ as $S$ is a path, so Maker has enough moves as we assumed $m\ge 2$. Let $f\in E(S)$ be an uncolored edge adjacent to $e$. We have that $f$ is adjacent to at most two edges in $E(S)$, both of which could already be colored; $f$ is also adjacent to some edges, including $e$, which are not in $E(S)$ --- but among these edges, only $e$ is colored --- as if there is another colored edge $e'$, then $f$ should have been colored by Maker right after $e'$ was colored by Breaker. So among all edges adjacent to $f$, at most $2+1=3$ of them are already colored, and because there are $\Delta(T)\ge 4$ colors, Maker can find a feasible color for $f$. Now we have showed that Maker can color the one or two edges in $E(S)$ adjacent to $e$. And then Maker will use the remaining $m-1$ or $m-2$ moves to color other uncolored edges in $E(S)$. In fact, at this point, any uncolored edge in $E(S)$ is adjacent to at most two colored edges, so Maker always has a feasible color for it.

    If Breaker colors an edge in $E(S)$, then Maker can randomly choose $m$ edges in $E(S)$ and properly color them as any uncolored edge in $E(S)$ is adjacent to at most two colored edges.

    Thus, at some point, all edges in $E(S)$ will be colored, and as mentioned at the beginning, this is sufficient to prove $\chi'_g(T;m,1)\le \Delta(T)$, hence $\chi'_g(T;m,1)=\Delta(T)$.
\end{proof}

If $T$ is a caterpillar with maximum degree $\Delta(T)=3$, then both Proposition \ref{trivial} and Theorem \ref{trees} tell us that $\chi'_g(T;m,1)\le 5$. However, using the proof we just made, we can see that $\chi'_g(T;m,1)\le 4$.

\begin{corollary}
    Let $T$ be a caterpillar and let $m\ge 2$ be an integer. If $\Delta(T)=3$, then
    \begin{align*}
        \chi'_g(T;m,1)\le 4.
    \end{align*}
\end{corollary}

\section{Wheels}
Naturally, one may conjecture that the more moves Maker has, the fewer colors we need, that is, $\chi'_g(G;m_1,1)\le \chi'_g(G;m_2,1)$ for any graph $G$ and any $m_1>m_2$. However, this is not true! In this section, we will show a counterexample, and give some intuitions about why this conjecture fails.

A \emph{wheel} with $n+1$ vertices, denoted by $W_n$, is constructed by connecting a vertex $v_0$ to every vertex in the cycle $C_n$. In $W_n$, an edge is called a \emph{spoke} if $v_0$ is on this edge, and an edge is called a \emph{rim edge} if $v_0$ is not on this edge.

Andres, Hochst\"attler, and Schall\"uck \cite{AHS} determined the exact game chromatic indices of wheels by proving the following theorem and observing that $\chi'_g(W_3)=\chi'_g(W_4)=5$ and $\chi'_g(W_5)=6$.

\begin{theorem}[Andres, Hochst\"attler, and Schall\"uck \cite{AHS}]
    Let $W_n$ be the wheel with $n+1$ vertices. If $n\ge 6$, then we have
    \begin{align*}
        \chi'_g(W_n)=n.
    \end{align*}
\end{theorem}

So the $(1,1)$-game chromatic indices of wheels are already determined.

For the $(m,1)$-edge coloring game with $m\ge 2$, we first handle the case that $n$ is small.

\begin{theorem} \label{smallwheels}
    Let $W_n$ be the wheel with $n+1$ vertices and let $m\ge 2$ be an integer. We have
    \begin{align*}
        \chi'_g(W_3;m,1)=3,
    \end{align*}
    and
    \[\chi'_g(W_4;m,1)=\begin{cases}
        4 & if\ m\neq 3, \\
        5 & if\ m=3.
    \end{cases}\]
\end{theorem}

Note that we have $\chi'_g(W_4;3,1)>\chi'_g(W_4;2,1)$, which is a counterexample to the conjecture at the beginning of the section.

\begin{proof}
    Firstly, by Proposition \ref{trivial}, we know that $\chi'_g(W_3;m,1)\ge 3$ and $\chi'_g(W_4;m,1)\ge 4$.
    
    By giving Maker a strategy, we show that $\chi'_g(W_3;m,1)\le 3$. For $W_3$, as showed in Figure \ref{W3}, we let the three spokes be $s_i$ and let the three rim edges be $r_i$ for $1\le i\le 3$.
    \begin{figure}[H]
        \tikzset{every picture/.style={line width=0.75pt}} %set default line width to 0.75pt        

\begin{tikzpicture}[x=0.75pt,y=0.75pt,yscale=-1,xscale=1]
%uncomment if require: \path (0,355); %set diagram left start at 0, and has height of 355

%Shape: Regular Polygon [id:dp9184152681172175] 
\draw   (386,190) -- (278,252.35) -- (278,127.65) -- cycle ;
%Straight Lines [id:da48536619970452644] 
\draw    (278,127.65) -- (314,190) ;
%Straight Lines [id:da5449210939466966] 
\draw    (278,252.35) -- (314,190) ;
%Straight Lines [id:da8000002508343065] 
\draw    (314,190) -- (386,190) ;
%Shape: Circle [id:dp9155631943886404] 
\draw  [fill={rgb, 255:red, 0; green, 0; blue, 0 }  ,fill opacity=1 ] (272.5,127.65) .. controls (272.5,124.61) and (274.96,122.15) .. (278,122.15) .. controls (281.04,122.15) and (283.5,124.61) .. (283.5,127.65) .. controls (283.5,130.68) and (281.04,133.15) .. (278,133.15) .. controls (274.96,133.15) and (272.5,130.68) .. (272.5,127.65) -- cycle ;
%Shape: Circle [id:dp5230788644377433] 
\draw  [fill={rgb, 255:red, 0; green, 0; blue, 0 }  ,fill opacity=1 ] (308.5,190) .. controls (308.5,186.96) and (310.96,184.5) .. (314,184.5) .. controls (317.04,184.5) and (319.5,186.96) .. (319.5,190) .. controls (319.5,193.04) and (317.04,195.5) .. (314,195.5) .. controls (310.96,195.5) and (308.5,193.04) .. (308.5,190) -- cycle ;
%Shape: Circle [id:dp08069679323120438] 
\draw  [fill={rgb, 255:red, 0; green, 0; blue, 0 }  ,fill opacity=1 ] (272.5,252.35) .. controls (272.5,249.32) and (274.96,246.85) .. (278,246.85) .. controls (281.04,246.85) and (283.5,249.32) .. (283.5,252.35) .. controls (283.5,255.39) and (281.04,257.85) .. (278,257.85) .. controls (274.96,257.85) and (272.5,255.39) .. (272.5,252.35) -- cycle ;
%Shape: Circle [id:dp8237866728690764] 
\draw  [fill={rgb, 255:red, 0; green, 0; blue, 0 }  ,fill opacity=1 ] (380.5,190) .. controls (380.5,186.96) and (382.96,184.5) .. (386,184.5) .. controls (389.04,184.5) and (391.5,186.96) .. (391.5,190) .. controls (391.5,193.04) and (389.04,195.5) .. (386,195.5) .. controls (382.96,195.5) and (380.5,193.04) .. (380.5,190) -- cycle ;

% Text Node
\draw (299,153) node [anchor=north west][inner sep=0.75pt]   [align=left] {$\displaystyle s_{1}$};
% Text Node
\draw (287,198) node [anchor=north west][inner sep=0.75pt]   [align=left] {$\displaystyle s_{2}$};
% Text Node
\draw (333,192) node [anchor=north west][inner sep=0.75pt]   [align=left] {$\displaystyle s_{3}$};
% Text Node
\draw (326,227) node [anchor=north west][inner sep=0.75pt]   [align=left] {$\displaystyle r_{1}$};
% Text Node
\draw (331,145) node [anchor=north west][inner sep=0.75pt]   [align=left] {$\displaystyle r_{2}$};
% Text Node
\draw (258,186) node [anchor=north west][inner sep=0.75pt]   [align=left] {$\displaystyle r_{3}$};

\end{tikzpicture}
\caption{Spokes and rim edges in $W_3$.}
\label{W3}
    \end{figure}
    If $m=2$, then on the first turn, Maker can color $s_1$ and $r_1$ using the the same color $A$. Breaker has to color an edge with a new color $B$. Then if Breaker uses $B$ on $s_i$, Maker will use $B$ on $r_i$; if Breaker uses $B$ on $r_i$, Maker will use $B$ on $s_i$. The last two edges can and must be colored with the third color.

    If $m\ge 3$, then on the first turn, Maker can color all three spokes using three different colors, and the colors of the rim edges will be fixed.
    
    So indeed we have $\chi'_g(W_3;m,1)\le 3$, and thus $\chi'_g(W_3;m,1)=3$.

    For $W_4$, as showed in Figure \ref{W4} (a), we let the four spokes be $s_i$ and let the four rim edges be $r_i$ for $1\le i\le 4$.
    \begin{figure}[H]
        \tikzset{every picture/.style={line width=0.75pt}} %set default line width to 0.75pt        

\begin{tikzpicture}[x=0.75pt,y=0.75pt,yscale=-1,xscale=1]
%uncomment if require: \path (0,463); %set diagram left start at 0, and has height of 463

%Shape: Regular Polygon [id:dp7827158615658734] 
\draw   (376,93) -- (307,162) -- (238,93) -- (307,24) -- cycle ;
%Straight Lines [id:da42019720067231314] 
\draw    (307,24) -- (307,93) ;
%Straight Lines [id:da6690038444475543] 
\draw    (307,93) -- (307,162) ;
%Straight Lines [id:da31494006415266207] 
\draw    (238,93) -- (307,93) ;
%Straight Lines [id:da6102648842131524] 
\draw    (307,93) -- (376,93) ;
%Shape: Circle [id:dp22508607448494367] 
\draw  [fill={rgb, 255:red, 0; green, 0; blue, 0 }  ,fill opacity=1 ] (301.5,24) .. controls (301.5,20.96) and (303.96,18.5) .. (307,18.5) .. controls (310.04,18.5) and (312.5,20.96) .. (312.5,24) .. controls (312.5,27.04) and (310.04,29.5) .. (307,29.5) .. controls (303.96,29.5) and (301.5,27.04) .. (301.5,24) -- cycle ;
%Shape: Circle [id:dp8040688280307766] 
\draw  [fill={rgb, 255:red, 0; green, 0; blue, 0 }  ,fill opacity=1 ] (301.5,93) .. controls (301.5,89.96) and (303.96,87.5) .. (307,87.5) .. controls (310.04,87.5) and (312.5,89.96) .. (312.5,93) .. controls (312.5,96.04) and (310.04,98.5) .. (307,98.5) .. controls (303.96,98.5) and (301.5,96.04) .. (301.5,93) -- cycle ;
%Shape: Circle [id:dp5815787380684161] 
\draw  [fill={rgb, 255:red, 0; green, 0; blue, 0 }  ,fill opacity=1 ] (232.5,93) .. controls (232.5,89.96) and (234.96,87.5) .. (238,87.5) .. controls (241.04,87.5) and (243.5,89.96) .. (243.5,93) .. controls (243.5,96.04) and (241.04,98.5) .. (238,98.5) .. controls (234.96,98.5) and (232.5,96.04) .. (232.5,93) -- cycle ;
%Shape: Circle [id:dp6847212758772374] 
\draw  [fill={rgb, 255:red, 0; green, 0; blue, 0 }  ,fill opacity=1 ] (370.5,93) .. controls (370.5,89.96) and (372.96,87.5) .. (376,87.5) .. controls (379.04,87.5) and (381.5,89.96) .. (381.5,93) .. controls (381.5,96.04) and (379.04,98.5) .. (376,98.5) .. controls (372.96,98.5) and (370.5,96.04) .. (370.5,93) -- cycle ;
%Shape: Circle [id:dp10450836082760384] 
\draw  [fill={rgb, 255:red, 0; green, 0; blue, 0 }  ,fill opacity=1 ] (301.5,162) .. controls (301.5,158.96) and (303.96,156.5) .. (307,156.5) .. controls (310.04,156.5) and (312.5,158.96) .. (312.5,162) .. controls (312.5,165.04) and (310.04,167.5) .. (307,167.5) .. controls (303.96,167.5) and (301.5,165.04) .. (301.5,162) -- cycle ;
%Shape: Regular Polygon [id:dp47438403400360896] 
\draw   (569,93) -- (500,162) -- (431,93) -- (500,24) -- cycle ;
%Straight Lines [id:da10793714268328491] 
\draw    (500,24) -- (500,93) ;
%Straight Lines [id:da473102344108973] 
\draw    (500,93) -- (500,162) ;
%Straight Lines [id:da8267351021497746] 
\draw    (431,93) -- (500,93) ;
%Straight Lines [id:da5297928066447666] 
\draw    (500,93) -- (569,93) ;
%Shape: Circle [id:dp987548191470909] 
\draw  [fill={rgb, 255:red, 0; green, 0; blue, 0 }  ,fill opacity=1 ] (494.5,24) .. controls (494.5,20.96) and (496.96,18.5) .. (500,18.5) .. controls (503.04,18.5) and (505.5,20.96) .. (505.5,24) .. controls (505.5,27.04) and (503.04,29.5) .. (500,29.5) .. controls (496.96,29.5) and (494.5,27.04) .. (494.5,24) -- cycle ;
%Shape: Circle [id:dp12212185755953642] 
\draw  [fill={rgb, 255:red, 0; green, 0; blue, 0 }  ,fill opacity=1 ] (494.5,93) .. controls (494.5,89.96) and (496.96,87.5) .. (500,87.5) .. controls (503.04,87.5) and (505.5,89.96) .. (505.5,93) .. controls (505.5,96.04) and (503.04,98.5) .. (500,98.5) .. controls (496.96,98.5) and (494.5,96.04) .. (494.5,93) -- cycle ;
%Shape: Circle [id:dp8051376554965657] 
\draw  [fill={rgb, 255:red, 0; green, 0; blue, 0 }  ,fill opacity=1 ] (425.5,93) .. controls (425.5,89.96) and (427.96,87.5) .. (431,87.5) .. controls (434.04,87.5) and (436.5,89.96) .. (436.5,93) .. controls (436.5,96.04) and (434.04,98.5) .. (431,98.5) .. controls (427.96,98.5) and (425.5,96.04) .. (425.5,93) -- cycle ;
%Shape: Circle [id:dp00044963567779388036] 
\draw  [fill={rgb, 255:red, 0; green, 0; blue, 0 }  ,fill opacity=1 ] (563.5,93) .. controls (563.5,89.96) and (565.96,87.5) .. (569,87.5) .. controls (572.04,87.5) and (574.5,89.96) .. (574.5,93) .. controls (574.5,96.04) and (572.04,98.5) .. (569,98.5) .. controls (565.96,98.5) and (563.5,96.04) .. (563.5,93) -- cycle ;
%Shape: Circle [id:dp135370007088647] 
\draw  [fill={rgb, 255:red, 0; green, 0; blue, 0 }  ,fill opacity=1 ] (494.5,162) .. controls (494.5,158.96) and (496.96,156.5) .. (500,156.5) .. controls (503.04,156.5) and (505.5,158.96) .. (505.5,162) .. controls (505.5,165.04) and (503.04,167.5) .. (500,167.5) .. controls (496.96,167.5) and (494.5,165.04) .. (494.5,162) -- cycle ;

% Text Node
\draw (308,53) node [anchor=north west][inner sep=0.75pt]   [align=left] {$\displaystyle s_{1}$};
% Text Node
\draw (269,79) node [anchor=north west][inner sep=0.75pt]   [align=left] {$\displaystyle s_{2}$};
% Text Node
\draw (290,112) node [anchor=north west][inner sep=0.75pt]   [align=left] {$\displaystyle s_{3}$};
% Text Node
\draw (327,96) node [anchor=north west][inner sep=0.75pt]   [align=left] {$\displaystyle s_{4}$};
% Text Node
\draw (256,127) node [anchor=north west][inner sep=0.75pt]   [align=left] {$\displaystyle r_{1}$};
% Text Node
\draw (341,127) node [anchor=north west][inner sep=0.75pt]   [align=left] {$\displaystyle r_{2}$};
% Text Node
\draw (341,47) node [anchor=north west][inner sep=0.75pt]   [align=left] {$\displaystyle r_{3}$};
% Text Node
\draw (253,47) node [anchor=north west][inner sep=0.75pt]   [align=left] {$\displaystyle r_{4}$};
% Text Node
\draw (501,53) node [anchor=north west][inner sep=0.75pt]   [align=left] {$\displaystyle A$};
% Text Node
\draw (462,79) node [anchor=north west][inner sep=0.75pt]   [align=left] {$\displaystyle B$};
% Text Node
\draw (483,112) node [anchor=north west][inner sep=0.75pt]   [align=left] {$\displaystyle C$};
% Text Node
\draw (520,96) node [anchor=north west][inner sep=0.75pt]   [align=left] {$\displaystyle D$};
% Text Node
\draw (449,127) node [anchor=north west][inner sep=0.75pt]   [align=left] {$\displaystyle A$};
% Text Node
\draw (534,127) node [anchor=north west][inner sep=0.75pt]   [align=left] {$\displaystyle B$};
% Text Node
\draw (534,47) node [anchor=north west][inner sep=0.75pt]   [align=left] {$\displaystyle C$};
% Text Node
\draw (446,47) node [anchor=north west][inner sep=0.75pt]   [align=left] {$\displaystyle D$};
% Text Node
\draw (295,182) node [anchor=north west][inner sep=0.75pt]   [align=left] {(a)};
% Text Node
\draw (490,182) node [anchor=north west][inner sep=0.75pt]   [align=left] {(b)};

\end{tikzpicture}
\caption{(a) Spokes and rim edges. (b) The only possible edge coloring.}
\label{W4}
    \end{figure}

    It is easy to check that, up to rotation, reflection, and renaming the colors, there is only one way to arrange four colors on the edges of $W_4$, which is showed in Figure \ref{W4} (b). So Maker's goal is to achieve this pattern, and Breaker's goal is to break this pattern.

    If $m=2$, we will show that four colors $A$, $B$, $C$, and $D$ are enough. On the first turn, Maker uses color $A$ on $s_1$ and $r_1$. Then Breaker has to use a new color $B$, and there are six possible choices. According to Breaker's choice, Maker can properly respond on the second turn. We show all six possible choices for Breaker and how Maker should respond in Figure \ref{6cases}, where "$C_{m2}$" means this edge is colored to be $C$ by Maker on the second turn, and "$B_{b1}$" means this edge is colored to be $B$ by Breaker on the first turn. After Maker finishes the second turn, everything is clear, because each of the three uncolored edges can only be colored by a specific color. Eventually we will get a pattern essentially the same as Figure \ref{W4} (b), so four colors are enough.

    \begin{figure}[H]
        \input 6cases.tex
    \end{figure}

    If $m=3$, we first show that four colors are not enough. Assume we only have four colors $A$, $B$, $C$, and $D$. On the first turn, Maker needs to use at least two colors, because for any three edges, two of them must be adjacent. If Maker only uses two colors, say $A$ and $B$, then to ensure that the pattern in Figure \ref{W4} (b) can be achieved, a color needs to be used on both $s_i$ and $r_i$ for some $1\le i\le 4$. Without loss of generality, we may assume $A$ is used on $s_1$ and $r_1$. Then if $B$ is used on $s_j$ for $j\neq 1$, Breaker can apply $C$ to $r_j$; if $B$ is used on $r_j$ for $j\neq 1$, Breaker can apply $C$ to $s_j$. In either case, the pattern in Figure \ref{W4} (b) is already broken, so four colors are not enough. If, on the first turn, Maker uses three colors, say $A$, $B$, and $C$, then we make the following claim.

    \textbf{Claim.} If a rim edge is colored, then the rim edge on the opposite side must also be colored.
    
    \begin{proof}[Proof of Claim]
        Let us assume that Maker uses color $A$ on $r_1$, does not color $r_3$, and uses $B$ and $C$ on two edges which are not $r_1$ or $r_3$. Then Breaker can use $A$ on $r_3$, and now the pattern in Figure \ref{W4} (b) is already broken.
    \end{proof}

    This claim helps us eliminate many possible cases, and now, up to rotation, reflection, and renaming the colors, there are only two cases to check.
    \begin{itemize}
        \item The first case is that Maker uses $A$ on $s_1$, $B$ on $s_2$, and $C$ on $s_3$. In this case, Breaker can use $D$ on $r_2$, and then $s_4$ is adjacent to edges with all four colors, it cannot be colored.
        \item The second case is that Maker uses $A$ on $r_1$, $B$ on $r_3$, and $C$ on $s_1$. In this case,
        Breaker can use $D$ on $s_4$, and then $s_2$ and $s_3$ must be colored by $B$, but only one of them can be colored by $B$ as $s_2$ and $s_3$ are adjacent, so the pattern in Figure \ref{W4} (b) cannot be achieved.
    \end{itemize}
    
    So, for $m=3$, four colors are not enough. Now assume that we have five colors $A$, $B$, $C$, $D$, and $E$. On the first turn, Maker can use $A$ on $s_1$, $B$ on $s_2$, and $C$ on $s_3$. If Breaker uses a new color $D$ on $s_4$, then the four spokes are colored distinctly, and there is always a feasible color for each rim edge, as a rim edge is only adjacent to four edges, and we have five colors. If Breaker reuses a color in $\{A,\ B,\ C\}$ or uses a new color $D$ but not on $s_4$, then Maker can use $E$ on $s_4$, so the four spokes are colored distinctly and then we can find a feasible color for each rim edge. Thus, five colors are enough, and we can conclude that $\chi'_g(W_4;3,1)=5$.

    If $m=4$, then four colors are enough. On the first turn, Maker can use color $A$ on $s_1$ and $r_1$, and use color $B$ on $s_2$ and $r_2$. Then Breaker has to use a new color $C$ on $s_i$ (or $r_i$) with $i\in \{3,\ 4\}$, and Maker can use $C$ on $r_i$ (or $s_i$) and use $D$ on the remaining two edges.

    If $m\ge 5$, then four colors are enough. With the first five moves of the first turn, Maker can use color $A$ on $s_1$ and $r_1$, color $B$ on $s_2$, color $C$ on $r_3$, and color $D$ on $s_4$. Then each of the remaining edges can only be colored by a specific color, that is, $s_3$ can only be colored by $C$, $r_2$ can only be colored by $B$, and $r_4$ can only be colored by $D$.
\end{proof}

In this theorem, we have already seen a counterexample to the conjecture we mentioned at the beginning of the section. Now we determine $\chi'_g(W_n;m,1)$ for $n\ge 5$. Some steps in the proof help explain why that conjecture fails.

\begin{theorem} \label{wheels}
    Let $W_n$ be the wheel with $n+1$ vertices and let $m\ge 2$ be an integer. If $n\ge 5$, then we have 
    \begin{align*}
        \chi'_g(W_n;m,1)=n.
    \end{align*}
\end{theorem}

\begin{proof}
    Assume $n\ge 5$. By Proposition \ref{trivial}, we have $\chi'_g(W_n;m,1)\ge \Delta(W_n)=n$. To show that $\chi'_g(W_n;m,1)\le n$, we need to prove that Maker has a winning strategy with $n$ colors. And it suffices to show that if we have $n$ colors, then whatever Breaker does, Maker can always get all spokes colored. This is because there is always a feasible color for each rim edge, as a rim edge is adjacent to four edges, and we have $n\ge 5$ colors.

    Maker can start with randomly coloring spokes. Note that spokes must receive distinct colors, as all spokes are adjacent to each other. Before the $(n-3)$-th spoke is colored:
    \begin{itemize}
        \item If Breaker colors a spoke, then it only helps Maker color one more spoke, and Maker can continue coloring more spokes.
        \item If Breaker colors a rim edge with a used color, then Maker can continue coloring more spokes.
        \item If Breaker colors a rim edge $r\in E(W_n)$ with a new color, then Maker first needs to use the same color on a spoke --- this is doable because there are at least four uncolored spokes, and $r$ forbids at most two of them to be colored the same as $r$. And then Maker can continue coloring other spokes.
    \end{itemize}

    These strategies help Maker guarantee that before the $(n-3)$-th spoke is colored, if exactly $k$ spokes have been colored, then exactly $k$ colors have been used.

    When the $(n-3)$-th spoke is colored, we assume the three uncolored spokes are $s,s',s''$. Depending on who colors the $(n-3)$-th spoke, we have two cases.

    \textbf{Case 1.} The $(n-3)$-th spoke is colored by Breaker.

    \textbf{Subcase 1.1.} $m\ge 3$. 
    
    Maker can color $s,s',s''$ with three more colors, now all spokes are colored and exactly $n$ colors are used.

    \textbf{Subcase 1.2.} $m=2$.
    
    Maker will color one spoke and make the remaining two uncolored spokes not next to each other. There are three possibilities:
    \begin{itemize}
        \item If $s$ is next to both $s'$ and $s''$ (which means $s$ is in the middle of $s'$ and $s''$), then Maker can color $s$.
        \item If $s$ is next to $s'$, but $s''$ is not next to $s$ or $s'$, then Maker can color $s$ or $s'$.
        \item If none of $s,s',s''$ is next to another, then Maker can randomly color one of them.
    \end{itemize}
    
    After coloring one spoke, Maker will use the second move to randomly color a rim edge with a used color. This is doable because Maker has only colored spokes before this move, and Maker has more moves than Breaker in each turn, which means there must be at least one uncolored rim edge as they have taken the same number of turns. Note that in Subcase 2.4., when Maker colors a rim edge with a used color, it is doable because of the same reason.

    Now it is Breaker's turn, there are two uncolored spokes which are not next to each other, and so far exactly $n-2$ colors have been used. 
    \begin{itemize}
        \item If Breaker colors a spoke, then Maker can color the other one. They use two more colors.
        \item If Breaker colors a rim edge with a used color, then Maker can color both spokes with two more colors.
        \item If Breaker colors a rim edge with a new color, then Maker can use this color on one of the two spokes (because the two spokes are not next to each other), and then use one more color on the last spoke.
    \end{itemize}
    
    \textbf{Case 2.} The $(n-3)$-th spoke is colored by Maker.

    \textbf{Subcase 2.1.} Maker has at least three more moves on this turn.

    In this subcase, Maker can color $s,s',s''$ on this turn. And exactly $n$ colors are used when all spokes are colored.

    \textbf{Subcase 2.2.} Maker has two more moves on this turn.

    This subcase is the same as Subcase 1.2.

    \textbf{Subcase 2.3.} Maker has one more move on this turn.

    This subcase is essentially the same as Subcase 1.2., the only difference is that Maker will only color a spoke, and not have the second move to color a rim edge with a used color.

    \textbf{Subcase 2.4.} Maker has no more moves on this turn.

    So it will be Breaker's turn with three uncolored spokes $s,s',s''$ left, and so far $n-3$ colors have been used.
    \begin{itemize}
        \item If Breaker colors a spoke, then Maker can color the other two. They use three more colors.
        \item If Breaker colors a rim edge with a used color, then the situation can be handled the same as Case 1.
        \item If Breaker colors a rim edge $r\in E(W_n)$ with a new color, and $m\ge 3$, then Maker can use the same color on a spoke not adjacent to $r$, and color the other two spokes with two more colors.
        \item If Breaker colors a rim edge $r\in E(W_n)$ with a new color, and $m=2$, then:
        \begin{itemize}
            \item If $r$ is adjacent to at most one of the three uncolored spokes, then Maker can use the color of $r$ on a spoke not adjacent to $r$ and make the remaining two spokes not next to each other. This is obviously doable if none of $s,s',s''$ is next to another. And if $s$ is next to both $s'$ and $s''$, then $r$ is only adjacent to at most one of $s'$ and $s''$, and Maker can color $s$ with the color of $r$. If $s$ is next to $s'$, but $s''$ is not next to $s$ or $s'$, then Maker can color one of $s$ and $s'$ as $r$ cannot be adjacent to both of them. Then Maker's second move is coloring a rim edge with a used color just like what we have in Subcase 1.2., and the rest of the proof is the same as Subcase 1.2.
            \item If $r$ is adjacent to two of the three uncolored spokes, then Maker can use the color of $r$ on the spoke not adjacent to $r$, and then color a rim edge with a used color just like what we have in Subcase 1.2. Now we have used $n-2$ colors. Then it will be Breaker's turn, but Breaker cannot color a rim edge adjacent to both uncolored spokes, because the only such rim edge is $r$ and it is already colored. If Breaker colors a spoke, then Maker can color the other spoke, and they use two more colors. If Breaker colors a rim edge with a used color, then Maker can color both spokes with two more colors. If Breaker colors a rim edge $r'\in E(W_n)$, which is adjacent to at most one uncolored spoke, with a new color, then Maker can apply this color on the uncolored spoke not adjacent to $r'$, and color the last spoke with one more color.
        \end{itemize} 
    \end{itemize}
    
Now we have showed that if we have $n$ colors, then whatever Breaker chooses to do, Maker can always get all spokes colored. And as mentioned at the very beginning of the proof, this is suffient to show that $\chi'_g(W_n;m,1)=n$ for $m\ge 2$.
\end{proof}

Let us give some intuitions about why we do not always have $\chi'_g(G;m_1,1)\le \chi'_g(G;m_2,1)$ for $m_1>m_2$. This is because in some cases Maker may not need that many moves, but still needs to use up all moves. For example, in Subcase 1.2. of the proof of Theorem \ref{wheels}, Maker needs to use the second move to color a rim edge with a used color, which is essentially because Maker does not need the second move but still needs to use it. Also, in the proof of Theorem \ref{smallwheels}, Maker wants to use the same color on $s_i$ and $r_i$ for each $1\le i\le 4$, so we can see $s_i$ and $r_i$ as a pair, and if $m$ is even, then Maker can color $\frac{m}{2}$ pairs on the first turn, and observe what Breaker will do. But if $m=3$, then Maker cannot handle the first turn perfectly, which leads to us needing more colors.

\section{Remarks}
For some specific $m_1>m_2$, we do have $\chi'_g(G;m_1,1)\le \chi'_g(G;m_2,1)$ for any $G$.

\begin{proposition}
    Let $m_1$ and $m_2$ be two positive integers. If there is a positive integer $k$ such that $m_1=km_2+k-1$, then we have
    \begin{align*}
        \chi'_g(G;m_1,1)\le \chi'_g(G;m_2,1)
    \end{align*}
    for any graph $G$. In particular, if $m_1$ is odd and $m_2=1$, then
    \begin{align*}
        \chi'_g(G;m_1,1)\le \chi'_g(G;1,1)=\chi'_g(G)
    \end{align*}
    for any graph $G$.
\end{proposition}

We omit the formal proof and briefly explain the idea. If $m_1=km_2+k-1$, then $m_1$ can be interpreted as "$k$ turns of Maker in the $(m_2,1)$-edge coloring game" plus "$k-1$ turns of Breaker in the $(m_2,1)$-edge coloring game". So in the $(m_1,1)$-edge coloring game, on each turn, Maker can use $k-1$ moves to play as Breaker, and create the same pattern as what they would have in the $(m_2,1)$-edge coloring game.

At the end of this paper, we make a bold conjecture.

\begin{conjecture}
    If $m_1>m_2$ and $m_1\neq km_2+k-1$ for any $k$, then there is a graph $G$ such that
    \begin{align*}
        \chi'_g(G;m_1,1)>\chi'_g(G;m_2,1).
    \end{align*}
\end{conjecture}

\section*{Acknowledgments}
The author thanks the anonymous reviewers for their insightful comments and valuable suggestions, which helped improve this paper a lot.

\end{document}